\documentclass{amsart}
\usepackage{amsfonts}

\newtheorem{theorem}{Theorem}
\theoremstyle{plain}

\newtheorem{corollary}{Corollary}

\newtheorem{definition}{Definition}

\newtheorem{proposition}{Proposition}
\newtheorem{remark}{Remark}

\numberwithin{equation}{section}

\input{tcilatex}

\begin{document}
\title{Some Properties of Mappings on Generalized Topological Spaces}
\author{Xun Ge}
\author{Jianhua Gong}
\author{Ivan Reilly}

\begin{abstract}
This paper considers generalizations of open mappings, closed mappings,
pseudo-open mappings, and quotient mappings from topological spaces to
generalized topological spaces. Characterizations of these classes of
mappings are obtained and some relationships among these classes are
established.
\end{abstract}

\maketitle

\section{Introduction}

Generalized topological spaces are an important generalization of
topological spaces. Let $X$ be a set and $\mu \subseteq \exp X$. Then $\mu $
is called a generalized topology on $X$ and $(X,\mu )$ is called a
generalized topological space, if $\mu $ satisfies the following two
properties:

(1) $\emptyset \in \mu$,

(2) Any union of elements of $\mu$ belongs to $\mu$.

There are plenty of generalized topological spaces which are not topological
spaces. For instance, let $X=\{a,b,c\}$ and put $\mu =\{ \emptyset
,\{a,b\},\{b,c\},X\}$, then $(X,\mu )$ is a generalized topological space
but it is not a topological space. Many interesting results and valuable
applications of generalized topological spaces were previously considered
(see \cite{C2, C1, GYC, GXY, La, M, SS, S}, for example).

Min investigated $(\mu ,\nu )$-continuous mappings from a generalized
topological space $(X,\mu )$ to another generalized topological space $%
(Y,\nu )$, which were generalized continuous mappings introduced by Csaszar
in \cite{C2}. He then obtained some interesting characterizations for $(\mu
,\nu )$-continuous mappings (\cite{M}). In this paper, we make further
investigations for $(\mu ,\nu )$-continuous mappings and develop some aspect
of mapping theory in generalized topology. We consider some mappings which
are defined on topological spaces for example open mappings, closed
mappings, pseudo-open mappings, and quotient mappings, and we generalize
them to the setting of generalized topological spaces (see Definition \ref{g
mappings}). We obtain characterizations of these classes of mappings, and
establish some relationships among these classes.

Throughout this paper, generalized topological spaces are denoted by $GT$%
-spaces and $(\mu ,\nu )$-continuous mappings are in fact a generalization
of continuous mappings and are denoted by $g$-continuous mappings (see
Definition \ref{g continuous}). $(X,\mu )$ and $(Y,\nu )$ always denote $GT$%
-spaces. Every mapping in this paper is assumed to be surjective, and every $%
GT$-space is assumed to be "strong". Here, a $GT$-space $(X,\mu )$ is said
to be strong if $X\in \mu $ \cite{C1}.

Let $(X,\mu )$ be a $GT$-space and $B\subseteq X$. $B$ is called $\mu $-open
(resp. $\mu $-closed) in $X$ if $B\in \mu $ (resp. $X$-$B\in \mu $). $I(B)$
denotes the largest $\mu $-open subset of $B$, i.e., $I(B)$ is the union of
all $\mu $-open subsets of $B$, and is called the interior of $B$ in $(X,\mu
)$. $C(B)$ denotes the smallest $\mu $-closed subrset of $B$, i.e., $C(B)$
is the intersection of all $\mu $-closed subsets of $B$, and is called the
closure of $B$ in $(X,\mu )$. For $x\in X,$ let $U_{x}$ be a $\mu $-open
subset containing $x,$ i.e., $x\in U_{x}\in \mu $. Then $x$ is called a $\mu 
$-cluster point of $B$ in $(X,\mu )$ if $U_{x}\bigcap (B-\{x\})\neq
\emptyset $ for each $U_{x}$.\ Let $d(B)$ denote the set of all $\mu $%
-cluster points of $B$ in $(X,\mu )$. Thus, $x\in d(B)$ if and only if $x\in
C(B-\{x\}).$

\section{Preliminaries}

The following notations are used throughout this paper. Let $(X,\mu )$ be a $%
GT$-space, let $x\in X$ and $\mathcal{F}\subseteq \exp X$.

(1) $\bigcup \mathcal{F}=\bigcup \{F:F\in \mathcal{F}\}$.

(2) $\bigcap \mathcal{F}=\bigcap \{F:F\in \mathcal{F}\}$.

(3) $\mu _{x}=\{U:x\in U\in \mu \}.$

(4) $N(x)=\bigcap \mu _{x}$.

\bigskip

In generalized topological spaces, some basic properties of interior
subsets, closure subsets, $\mu$-open subsets, $\mu$-closed subsets, and $\mu 
$-cluster points are given in the following proposition (see \cite{C2,C1,M},
for example).

\begin{proposition}
\label{properties of subsets}Let $B$ be a subset of a $GT-$space $(X,\mu )$.
Then the following hold.

(1) $I(B)\subseteq B\subseteq C(B)$.

(2) $I(I(B))=I(B)$ and $C(C(B))=C(B)$.

(3) If $B^{\prime }\subseteq B$, then $I(B^{\prime })\subseteq I(B)$, $%
C(B^{\prime })\subseteq C(B)$ and $d(B^{\prime })\subseteq d(B)$.

(4) $I(B)=B$ $\Longleftrightarrow $ $B$ is $\mu$-open in $X$ $%
\Longleftrightarrow $ for each $x\in B,U_{x}\subseteq B$ for some $U_{x}\in
\mu _{x}$.

(5) $C(B)=B$ $\Longleftrightarrow $ $B$ is $\mu $-closed in $X$ $%
\Longleftrightarrow $ for each $x\in X-B,U_{x}\bigcap B=\emptyset $ for some 
$U_{x}\in \mu _{x}$.

(6) $C(B)=X-I(X-B)$ and $I(B)=X-C(X-B).$

(7) $x\in C(B)$ $\Longleftrightarrow $ $U_{x}\bigcap B\neq \emptyset $ for
each $U_{x}\in \mu _{x}$.

(8) $x\in I(B)$ $\Longleftrightarrow $ $U_{x}\subseteq B$ for some $U_{x}\in
\mu _{x}$.

(9) $C(B)=B\bigcup d(B)$.

(10) $x\not \in d(\{x\})$ for each $x\in X$.
\end{proposition}

\begin{definition}
$\label{g continuous}$Let $f:(X,\mu )\longrightarrow (Y,\nu )$ be a mapping.
Then $f$ is called $g$-continuous \cite{C2} if $f^{-1}(U)\in \mu $ for each $%
U\in \nu $.
\end{definition}

Notice that $g$-continuous mappings here are the same as $(\mu ,\nu )$%
-continuous mappings. Thus, the characterizations of $g$-continuous mappings
in the next proposition are the characterizations of $(\mu ,\nu )$%
-continuous mappings obtained in \cite{M}.

\begin{proposition}
\label{characterizations of g continuous}Let $f:(X,\mu )\longrightarrow
(Y,\nu )$ be a mapping. Then the following are equivalent.

(1) $f$ is $g$-continuous.

(2) $f^{-1}(F)$ is $\mu $-closed in $X$ for each $\nu $-closed subset $F$ of 
$Y$.

(3) $f(C(A))\subseteq C(f(A))$ for each subset $A$ of $X$.

(4) $C(f^{-1}(B))\subseteq f^{-1}(C(B))$ for each subset $B$ of $Y$.

(5) $f^{-1}(I(B))\subseteq I(f^{-1}(B))$ for each subset $B$ of $Y$.

(6) For each $x\in X$, if $f(x)\in V\in \nu $, then $f(U)\subseteq V$ for
some $U\in \mu _{x}$.
\end{proposition}

Given a subset in a generalized topological space, we can also induce a
subspace of the generalized topological space. Furthermore, the relationship
between a generalized topological space and its subspaces regarding closure
and interior is discussed in the following proposition.

\begin{proposition}
\label{subspaces}Let $(X,\mu )$ be a $GT$-space. If $X^{\prime }$ is a
subset of $X,A$ is a subset of $X^{\prime },$ and put $\mu ^{\prime
}=\{U\bigcap X^{\prime }:U\in \mu \}.$ Then the following properties hold.

(1) $(X^{\prime },\mu ^{\prime })$ is a $GT$-space, which is called a
subspace of $(X,\mu )$.

(2) $C_{X^{\prime }}(A)=C(A)\bigcap X^{\prime }$, where $C_{X^{\prime }}(A)$
is the closure of $A$ in $(X^{\prime },\mu ^{\prime })$.

(3) $I(A)\subseteq I_{X^{\prime }}(A)\bigcap I(X^{\prime })$, where $%
I_{X^{\prime }}(A)$ is the interior of $A$ in $(X^{\prime },\mu ^{\prime })$.
\end{proposition}

\begin{proof}
(1) It is clear.

(2) Let $x\in C_{X^{\prime }}(A)$. For each $U\in \mu _{x}$, $U\bigcap
X^{\prime }\in \mu _{x}^{\prime }$. By Proposition \ref{properties of
subsets} (7), $(U\bigcap X^{\prime })\bigcap A\neq \emptyset $, and so $%
U\bigcap A\neq \emptyset $. Thus $x\in C(A)$. Note that $x\in X^{\prime }$.
So $x\in C(A)\bigcap X^{\prime }$. On the other hand, let $x\in C(A)\bigcap
X^{\prime },$ then $x\in X^{\prime }$ and $x\in C(A).$ For each $U^{\prime
}\in \mu _{x}^{\prime }$, there is $U\in \mu _{x}$ such that $U^{\prime
}=U\bigcap X^{\prime }$. By Proposition \ref{properties of subsets} (7), $%
U\bigcap A\neq \emptyset $. Notice that $A\subseteq X^{\prime }$. So $%
U^{\prime }\bigcap A=(U\bigcap X^{\prime })\bigcap A=U\bigcap A\neq
\emptyset .$ Thus $x\in C_{X^{\prime }}(A)$.

(3) From Proposition \ref{properties of subsets} (3), $I(A)\subseteq
I(X^{\prime })$ because $A\subseteq X^{\prime }$. So we only need to prove
that $I(A)\subseteq I_{X^{\prime }}(A)$. Let $x\in I(A)$, then there is $%
U\in \mu _{x}$ such that $x\in U\subseteq A\subseteq X^{\prime }$. It
suffices to prove that $x\in I_{X^{\prime }}(A)$. In fact, since $U\bigcap
X^{\prime }=U$, $U\in \mu _{x}^{\prime }$, and so $x\in I_{X^{\prime }}(A)$.
\end{proof}

\begin{remark}
(1) In Proposition \ref{subspaces} (3), \textquotedblleft $\subseteq $%
\textquotedblright \ can not be replaced by \textquotedblleft $=$%
\textquotedblright . In fact, let $X=\{x,y,z\}$, $\mu
=\{X,\{x,y\},\{y,z\},\phi \}$, $X^{\prime }=\{x,y\}$ and $\mu ^{\prime
}=\{X^{\prime },\{y\},\phi \}$. Then $(X^{\prime },\mu ^{\prime })$ is a
subspace of $(X,\mu )$. According to the definition, $I(\{y\})=$ the union
of all $\mu -$open subsets of $\{y\}=$ $\emptyset $. Similar, $I_{X^{\prime
}}(\{y\})=\{y\}$ and $I(X^{\prime })=\{x,y\}$. So $I_{X^{\prime
}}(\{y\})\bigcap I(X^{\prime })=\{y\} \bigcap \{x,y\}=\{y\} \neq I(\{y\})$.

(2 ) Given a subspace $(X^{\prime },\mu ^{\prime })$ of a $GT$-space $(X,\mu
), $ it is clear that $B$ is a $\mu ^{\prime }-$closed subset of $(X^{\prime
},\mu ^{\prime })$ if and only if there is a $\mu -$closed subset $F$ of $%
(X,\mu )$ such that $B=F\cap X^{\prime }.$
\end{remark}

\begin{definition}
Let $f:(X,\mu )\longrightarrow (Y,\nu )$ be a mapping and $X^{\prime
}\subseteq X$. Then $h:(X^{\prime },\mu ^{\prime })\longrightarrow
(f(X^{\prime }),\nu ^{\prime })$ is called a restriction of $f$ on $%
(X^{\prime },\mu ^{\prime })$, if $h(x)=f(x)$ for each $x\in X^{\prime }$,
where $\mu ^{\prime }=\{U\bigcap X^{\prime }:U\in \mu \}$ and $\nu ^{\prime
}=\{V\bigcap f(X^{\prime }):V\in \nu \}$. As usual, a restriction of $f$ on $%
(X^{\prime },\mu ^{\prime })$ is denoted by $f|_{X^{\prime }}$.
\end{definition}

\begin{proposition}
Let $f:(X,\mu )\longrightarrow (Y,\nu )$ be a mapping and $X^{\prime
}\subseteq X$. If $f$ is $g$-continuous, then its restriction $f|_{X^{\prime
}}$\ on $(X^{\prime },\mu ^{\prime })$ is $g$-continuous.
\end{proposition}

\begin{proof}
Let $f$ $:(X,\mu )\longrightarrow (Y,\nu )$ be $g$-continuous, $A\subseteq
X^{\prime }$, and the restriction $(X^{\prime },\mu ^{\prime })$ be $%
f|_{X^{\prime }}:(X^{\prime },\mu ^{\prime })\longrightarrow (f(X^{\prime
}),\nu ^{\prime }),$where $\mu ^{\prime }=\{U\bigcap X^{\prime }:U\in \mu \}$
and $\nu ^{\prime }=\{V\bigcap f(X^{\prime }):V\in \nu \}$. By Propositions %
\ref{subspaces} (2) and \ref{characterizations of g continuous} (3), so we
have $h(C_{X^{\prime }}(A))=h(C(A)\bigcap X^{\prime })=f(C(A)\bigcap
X^{\prime })\subseteq f(C(A))\bigcap f(X^{\prime })\subseteq C(f(A))\bigcap
f(X^{\prime })=C_{f(X^{\prime })}(f(A)),$where $C_{X^{\prime }}(A)$ and $%
C_{f(X^{\prime })}(f(A))$ are the closures of $A$ in $(X^{\prime },\mu
^{\prime })$ and $f(A)$ in $(f(X^{\prime }),\nu ^{\prime })$ respectively.
Thus $f|_{X^{\prime }}$ is $g$-continuous by Proposition \ref%
{characterizations of g continuous} (3).
\end{proof}

\bigskip

There are many useful mappings defined on topological spaces. Here we
consider the classes of open mappings \cite{E}, closed mappings \cite{E},
pseudo-open mappings \cite{GX}, and quotient mappings \cite{E}. We
generalize these classes of mappings on topological spaces to the classes of 
$g$-open mappings, $g$-closed mappings, $g$-pseudo-open mappings, and $g$%
-quotient mappings respectively, defined on generalized topological spaces
by the following definition.

\begin{definition}
\label{g mappings}Let $f:(X,\mu )\longrightarrow (Y,\nu )$ be a mapping. Then

(1) $f$ is called a $g$-open mapping if $f(U)\in \nu$ for each $U\in \mu$.

(2) $f$ is called a $g$-closed mapping if $f(F)$ is a $\nu$-closed subset of 
$Y$ for each $\mu$-closed subset $F$ of $X$.

(3) $f$ is called a $g$-pseudo-open mapping if for each $y\in Y$ and $%
f^{-1}(y)\subseteq U\in \mu $, then $y\in I(f(U))$.

(4) $f$ is called a $g$-quotient mapping if for each $V\subseteq Y$, $%
f^{-1}(V)\in \mu $ implies $V\in \nu $.
\end{definition}

We note that the generalized quotient map of a generalized quotient topology
(\cite{S}, Definition 3.1) is a $g$-quotient mapping according to our
Definition \ref{g mappings} (4). But a $g$-quotient mapping need not be a
generalized quotient map because that the corresponding generalized topology
need not be a generalized quotient topology. Thus, $g$-quotient mappings
here are generalizations of generalized quotient maps.

Furthermore, we generalize hereditary properties of mappings on topological
spaces \cite{E,L} to the context of generalized topological spaces.

\begin{definition}
\label{hereditarily}Let $f:(X,\mu )\longrightarrow (Y,\nu )$ be a mapping.
Then $f$ is called a hereditarily $g$-open mapping (resp. hereditarily $g$%
-closed mapping, hereditarily $g$-pseudo-open mapping, hereditarily $g$%
-quotient mapping) if $f|_{f^{-1}(Y^{\prime })}:(f^{-1}(Y^{\prime }),\mu
^{\prime })\longrightarrow (Y^{\prime },\nu ^{\prime })$ is $g$-open (resp. $%
g$-closed, $g$-pseudo-open, $g$-quotient) for each $Y^{\prime }\subseteq Y$.
Where $\mu ^{\prime }=\{U\bigcap f^{-1}(Y^{\prime }):U\in \mu \}$ and $\nu
^{\prime }=\{V\bigcap Y^{\prime }:V\in \nu \}.$
\end{definition}

\begin{proposition}
\label{hereditary property}Every $g$-open mapping (resp. $g$-closed mapping, 
$g$-pseudo-open mapping) is hereditarily $g$-open (resp. $g$-closed, $g$%
-pseudo-open).
\end{proposition}

\begin{proof}
Let $f:(X,\mu )\longrightarrow (Y,\nu )$ be a mapping. For each $Y^{\prime
}\subseteq Y$, put $X^{\prime }=f^{-1}(Y^{\prime })$, $\mu ^{\prime
}=\{U\bigcap X^{\prime }:U\in \mu \},\nu ^{\prime }=\{V\bigcap Y^{\prime
}:V\in \nu \},$ and $h=f|_{X^{\prime }}:(X^{\prime },\mu ^{\prime
})\longrightarrow (Y^{\prime },\nu ^{\prime })$. Notice that all mappings in
this paper are assumed to be surjective, hence $f(X^{\prime })=Y^{\prime }.$

(1) Assume that $f$ is $g$-open. Let $W\in \mu ^{\prime }$, then there is $%
U\in \mu $ such that $W=U\bigcap X^{\prime }$. So $h(W)=h(U\bigcap X^{\prime
})=f(U\bigcap f^{-1}(Y^{\prime }))=f(U)\bigcap f(f^{-1}(Y^{\prime
})))=f(U)\bigcap Y^{\prime }. $ Since $f$ is $g$-open, $f(U)\in \nu $, so $%
h(W)\in \nu ^{\prime }$. This shows that $h$ is $g$-open, so that $f$ is
hereditarily $g$-open.

(2) Assume that $f$ is $g$-closed. Let $F$ is a $\mu ^{\prime }$-closed
subset of $(X^{\prime },\mu ^{\prime })$. By Propositions \ref{properties of
subsets} (5) and \ref{subspaces} (2), there is a $\mu $-closed subset $E$ of 
$(X,\mu )$ such that $F=E\bigcap X^{\prime }$. So $h(F)=h(E\bigcap X^{\prime
})=f(E\bigcap f^{-1}(Y^{\prime }))=f(E)\bigcap Y^{\prime }$. Since $f$ is $g$%
-closed, $f(E)$ is a $\nu $-closed subset of $(Y,\nu )$, so $h(F)$ is a $\nu
^{\prime }$-closed subset of $(Y^{\prime },\nu ^{\prime })$. This shows that 
$h$ is $g$-closed, so that $f$ is hereditarily $g$-closed.

(3) Assume that $f$ is $g$-pseudo-open. Let $y\in Y^{\prime }$ and $%
h^{-1}(y)\subseteq W\in \mu ^{\prime }$, then there is $U\in \mu $ such that 
$W=U\bigcap X^{\prime }$. Since $h^{-1}(y)\subseteq W\subseteq X^{\prime }$
and $h=f|_{X^{\prime }},$ $f^{-1}(y)=h^{-1}(y)$. So $f^{-1}(y)\subseteq
W\subseteq U$. Since $f$ is $g$-pseudo-open, $y\in I(f(U)),$ hence $y\in
I(f(U))\bigcap Y^{\prime }.$ Since $I(f(U))\bigcap Y^{\prime }\in \nu
^{\prime },I(f(U))\bigcap Y^{\prime }=I_{Y^{\prime }}(I(f(U))\bigcap
Y^{\prime }). $ On the other hand, from Proposition \ref{properties of
subsets} (3), we have $I_{Y^{\prime }}(I(f(U))\bigcap Y^{\prime })\subseteq
I_{Y^{\prime }}(f(U)\bigcap Y^{\prime })=I_{Y^{\prime }}(f(U\bigcap
X^{\prime }))=I_{Y^{\prime }}(h(U\bigcap X^{\prime }))=I_{Y^{\prime }}(h(W))$
which gives $y\in I_{Y^{\prime }}(h(W)).$ This proves that $h$ is $g$%
-pseudo-open, so that $f$ is hereditarily $g$-pseudo-open.
\end{proof}

\begin{remark}
A $g$-quotient mapping need not be hereditarily $g$-quotient. In fact, we
will show [Theorem \ref{characterizations of g pseudo open}] that a mapping
is $g$-pseudo-open if and only if it is hereditarily $g$-quotient, but a $g$%
-quotient mapping need not be a $g$-pseudo-open mapping by Remark \ref%
{reverse relationship}.
\end{remark}

\section{The Main Results}

In this section, we obtain characterizations of $g-$open mappings, $g-$%
closed mappings, $g-$pseudo-open mappings, and $g-$quotient mappings on
generalized topological spaces.

\begin{theorem}
\label{characterizations of g open}Let $f:(X,\mu )\longrightarrow (Y,\nu )$
be a mapping. Then the following conditions are equivalent.

(1) $f$ is a $g$-open mapping.

(2) If $B\subseteq Y$, then $f^{-1}(C(B))\subseteq C(f^{-1}(B))$.

(3) If $A\subseteq X$, then $f(I(A))\subseteq I(f(A))$.

(4) For each $x\in X,$ if $x\in U\in \mu $, then $f(x)\in V\subseteq f(U)$
for some $V\in \nu $.
\end{theorem}

\begin{proof}
(1) $\Longrightarrow $ (2): Assume that $f$ is a $g$-open mapping.\ Let $%
x\in f^{-1}(C(B))$, then $f(x)\in ff^{-1}(C(B))=C(B)$. On the other hand,
whenever $U_{x}\in \mu _{x}$ implies $f(x)\in f(U_{x}),$ then $f(U_{x})\in
\nu _{f(x)}$ because $f$ is a $g$-open mapping. It follows from Proposition %
\ref{properties of subsets} (7) that $f(U_{x})\bigcap B\neq \emptyset .$
Choose $y\in f(U_{x})\bigcap B$. Then there is $x^{\prime }\in U_{x}$ such
that $y=f(x^{\prime })\in B$, i.e., $x^{\prime }\in f^{-1}(B)$. Therefore $%
x^{\prime }\in U\cap f^{-1}(B)\neq \emptyset $ which gives $x\in
C(f^{-1}(B)) $ by using Proposition \ref{properties of subsets} (7).

(2) $\Longrightarrow $ (3): Assume that the condition (2) holds. Let $%
A\subseteq X$. By Proposition \ref{properties of subsets} (1), $%
I(A)\subseteq A\subseteq f^{-1}f(A)$, so $I(A)\subseteq I(f^{-1}f(A))$. By
Proposition \ref{properties of subsets} (6), $%
I(f^{-1}f(A))=X-C(X-f^{-1}f(A))=X-C(f^{-1}(Y-f(A))),$ so $I(A)\subseteq
X-C(f^{-1}(Y-f(A))).$ Since the condition (2) holds, $%
f^{-1}(C(Y-f(A)))=C(f^{-1}(Y-f(A)))$, and hence $I(A)\subseteq
X-f^{-1}(C(Y-f(A)))=f^{-1}(Y-C(Y-f(A)))=f^{-1}(I(f(A)))$. It follows that $%
f(I(A))\subseteq ff^{-1}(I(f(A)))=I(f(A))$.

(3) $\Longrightarrow $ (1): Assume that the condition (3) holds. Let $U\in
\mu $, then $I(U)=U$, and $f(U)=f(I(U)\subseteq I(f(U))$. On the other hand, 
$I(f(U))\subseteq f(U)$ by using Proposition \ref{properties of subsets}
(1). Thus, $I(f(U))=f(U)$ which means $f(U)\in \nu $ by Proposition \ref%
{properties of subsets} (4). This proves that $f$ is a $g$-open mapping.

(1) $\Longrightarrow $ (4): Assume that $f$ is a $g$-open mapping. For each $%
x\in X,$ if $x\in U\in \mu $, then $f(x)\in f(U)\in \nu ,$ for $V=f(U).$

(4) $\Longrightarrow $ (1): Assume that the condition (4) holds. Let $U\in
\mu $. For each $y\in f(U)$, there is $x_{y}\in U$ such that $y=f(x_{y})$.
So there is $V_{y}\in \nu $ such that $y\in V_{y}\subseteq f(U)$. It follows
that $f(U)=\bigcup \{V_{y}:y\in f(U)\},$ and $\bigcup \{V_{y}:y\in f(U)\} \in
\nu $ because $Y$ is a $GT$. This proves that $f$ is a $g$-open mapping.
\end{proof}

From Theorem \ref{characterizations of g open} and Proposition \ref%
{characterizations of g continuous} above, the next Corollary holds
immediately.

\begin{corollary}
Let $f:(X,\mu )\longrightarrow (Y,\nu )$ be a $g$-continuous mapping. Then
the following conditions are equivalent.

(1) $f$ is a $g$-open mapping.

(2) If $B\subseteq Y$, then $f^{-1}(C(B))=C(f^{-1}(B))$.
\end{corollary}

The following theorem characterizing $g$-closed mappings in generalized
topological spaces is a generalization of the closed mapping theorem in
topological spaces (see \cite[Theorem 1.4.13]{E} or \cite[Lemma 2.3]{GY},
for example)

\begin{theorem}
\label{characterizations of g closed}Let $f:(X,\mu )\longrightarrow (Y,\nu )$
be a mapping. Then the following conditions are equivalent.

(1) $f$ is a $g$-closed mapping.

(2) If $A\subseteq X$, then $C(f(A))\subseteq f(C(A)).$

(3) If $B\subseteq Y$ and $U\in \mu $ such that $f^{-1}(B)\subseteq U$, then
there is $V\in \nu $ such that $B\subseteq V$ and $f^{-1}(V)\subseteq U$.

(4) If $y\in Y$ and $U\in \mu $ such that $f^{-1}(y)\subseteq U$, then there
is $V\in \nu $ such that $y\in V$ and $f^{-1}(V)\subseteq U$.
\end{theorem}

\begin{proof}
(1) $\Longrightarrow $ (2): Assume that $f$ is a $g$-closed mapping. Let $%
A\subseteq X$, then $C(A)$ is a $\mu $-closed subset of $X$, and so $f(C(A))$
is a $\nu $-closed subset of $Y$. Since $A\subseteq C(A)$, $f(A)\subseteq
f(C(A))$, hence $C(f(A))\subseteq f(C(A))$ by using Proposition \ref%
{properties of subsets} (3).

(2) $\Longrightarrow $ (1): Assume that the condition (2) holds. Let $A$ be
a $\mu $-closed subset of $X$. Then $C(A)=A$, and $f(A)=f(C(A))=C(f(A))$. By
Proposition \ref{properties of subsets} (5), $f(A)$ is a $\nu $-closed
subset of $Y$. This proves that $f$ is a $g$-closed mapping

(1) $\Longrightarrow $ (3): Assume that $f$ is a $g$-closed mapping. Let $%
B\subseteq Y$ and $U\in \mu $ such that $f^{-1}(B)\subseteq U$. Put $%
V=Y-f(X-U)$. It suffices to check the following three claims.

Claim 1. $V\in \nu $.

Since $X-U$ is a $\mu $-closed subset of $X$ and $f$ is $g$-closed, so $%
f(X-U)$ is a $\nu $-closed subset of $Y$. It follows that $V=Y-f(X-U)\in \nu 
$.

Claim 2. $B\subseteq V$.

$f^{-1}(B)\subseteq U$ implies $X-U\subseteq X-f^{-1}(B)$, hence $%
f(X-U)\subseteq f(X-f^{-1}(B))=f(f^{-1}(Y-B))=Y-B$. It follows that $%
B\subseteq Y-f(X-U)=V$.

Claim 3. $f^{-1}(V)\subseteq U$.

$f^{-1}(V)=f^{-1}(Y-f(X-U))=f^{-1}(Y)-f^{-1}(f(X-U))\subseteq X-(X-U)=U$.

(3) $\Longrightarrow $ (4): It is clear.

(4) $\Longrightarrow $ (1): Assume that the condition (3) holds. Let $F$ be
a $\mu $-closed subset of $X$. We only need to prove that $f(F)$ is a $\nu $%
-closed subset of $Y$.

Let $y\in Y-f(F),$ then $f^{-1}(y)\bigcap F=\emptyset $, i.e., $%
f^{-1}(y)\subseteq X-F.$ Notice that $X-F\in \mu ,$ then by (4) there is $%
V\in \nu $ such that $y\in V$ and $f^{-1}(V)\subseteq X-F,$ i.e., $%
f^{-1}(V)\bigcap F=\emptyset $. It follows that $V\bigcap
f(F)=ff^{-1}(V)\bigcap f(F)=f(f^{-1}(V)\bigcap F)=f(\emptyset )=\emptyset $.
By Proposition \ref{properties of subsets} (5), $f(F)$ is a $\mu $-closed
subset of $Y$.
\end{proof}

The next Corollary holds immediately from Theorem \ref{characterizations of
g closed} and Proposition \ref{characterizations of g continuous} above.

\begin{corollary}
Let $f:(X,\mu )\longrightarrow (Y,\nu )$ be a mapping. Then the following
conditions are equivalent.

(1) $f$ is a $g$-closed mapping.

(2) If $A\subseteq X$, then $C(f(A))=f(C(A))$.
\end{corollary}

We now discuss the characterization of $g$-quotient mappings. We then show
the characterization of $g$-pseudo-open mappings.

\begin{theorem}
\label{characterizations of g quotient}Let $f:(X,\mu )\longrightarrow (Y,\nu
)$ be a mapping. Then the following conditions are equivalent.

(1) $f$ is a $g$-quotient mapping.

(2) For each subset $F$ of $Y$, if $f^{-1}(F)$ is a $\mu $-closed subset of $%
(X,\mu ),$ then $F$ is a $\nu $-closed subset of $(Y,\nu ).$
\end{theorem}

\begin{proof}
(1) $\Longrightarrow $ (2): Assume that $f$ is a $g$-quotient mapping. Let $%
F\subseteq Y$ such that $f^{-1}(F)$ is a $\mu $-closed subset of $(X,\mu )$.
Then $f^{-1}(Y-F)=X-f^{-1}(F)\in \mu $. Since $f$ is a $g$-quotient mapping, 
$Y-F\in \nu $, and so $F$ is a $\nu $-closed subset of $(Y,\nu ).$

(2) $\Longrightarrow $ (1). It can be proved by the same method.
\end{proof}

\begin{theorem}
\label{characterizations of g pseudo open}Let $f:(X,\mu )\longrightarrow
(Y,\nu )$ be a mapping. Then the following conditions are equivalent.

(1) $f$ is a $g$-pseudo-open mapping.

(2) $f$ is a hereditarily $g$-quotient mapping.

(3) If $B\subseteq Y$, then $C(B)\subseteq f(C(f^{-1}(B))).$
\end{theorem}

\begin{proof}
(1) $\Longrightarrow $ (2): Assume that $f$ is a $g$-pseudo-open mapping.
Proposition \ref{hereditary property} states that every $g$-pseudo-open
mapping is hereditarily $g$-pseudo-open. It follows that $%
f|_{f^{-1}(Y^{\prime })}$ is $g$-pseudo-open for each $Y^{\prime }\subseteq Y
$. Thus, it suffices to prove that every $g$-pseudo-open mapping $f$ is a $g$%
-quotient mapping so that $f|_{f^{-1}(Y^{\prime })}$ is $g$-quotient for
each $Y^{\prime }\subseteq Y$.

Let $V\subseteq Y$ such that $f^{-1}(V)\in \mu $. If $y\in V$, then $%
f^{-1}(y)\subseteq f^{-1}(V)\in \mu .$ Notice that $f$ is $g$-pseudo-open,
hence $y\in I(ff^{-1}(V))=I(V)$ for each $y\in V.$ It follows that $V\in \nu 
$. This proves that $f$ is a $g$-quotient mapping.

(2) $\Longrightarrow $ (3): Assume that $f$ is a hereditarily $g$-quotient
mapping. Let $B\subseteq Y$ $\ $and let $y\in C(B)=B\cup (C(B)-B)$.

(i) If $y\in B$, then $y\in B=f(f^{-1}(B))\subseteq f(C(f^{-1}(B)))$.

(ii) If $y\in C(B)-B$, let $Y^{\prime }=B\bigcup \{y\}$ and $\nu ^{\prime
}=\{V\bigcap Y^{\prime }:V\in \nu \},$ and let $X^{\prime }=f^{-1}(Y^{\prime
})=f^{-1}(B)\bigcup f^{-1}(y)$, $\mu ^{\prime }=\{U\bigcap X^{\prime }:U\in
\mu \}$ and $h=f|_{X^{\prime }}:(X^{\prime },\mu ^{\prime })\longrightarrow
(Y^{\prime },\nu ^{\prime })$, then $h$ is a $g$-quotient mapping because $f$
is hereditarily $g$-quotient. Since $C(B)$ is the smallest $\mu -$closed
subset of $(X,\mu )$ containing $B$ and $C(B)\cap Y^{\prime }=B\bigcup \{y\}$%
, $B$ is not a $g$-closed subset of $(Y^{\prime },\nu ^{\prime }),$ it
follows from Theorem \ref{characterizations of g quotient} that $h^{-1}(B)$
is not a $g$-closed subset of $(X^{\prime },\mu ^{\prime })$, and hence
there is $x\in C_{X^{\prime }}(h^{-1}(B))-h^{-1}(B).$ Notice that $%
h^{-1}(B)=f^{-1}(B), $ so $C_{X^{\prime }}(h^{-1}(B))=C_{X^{\prime
}}(f^{-1}(B))\subseteq C(f^{-1}(B))$ which gives $x\in C(f^{-1}(B))$, and
hence $f(x)\in f(C(f^{-1}(B))).$ On the other hand, since $C_{X^{\prime
}}(h^{-1}(B))\subseteq X^{\prime }$ and $h^{-1}(B)=f^{-1}(B),x\in
C_{X^{\prime }}(h^{-1}(B))-h^{-1}(B)\subseteq X^{\prime
}-f^{-1}(B)=f^{-1}(B)\bigcup f^{-1}(y)-f^{-1}(B)=f^{-1}(y)$, and hence $%
y=f(x)\in f(C(f^{-1}(B)))$. Finally, (i) and (ii) imply $C(B)\subseteq
f(C(f^{-1}(B)))$.

(3) $\Longrightarrow $ (1): Assume that the condition (3) holds. Let $y\in Y$
and $f^{-1}(y)\subseteq U\in \mu $. Then $f^{-1}(y)\bigcap (X-U)=\emptyset $%
, so $y\not \in f(X-U)$, i.e., $y\in Y-f(X-U)$. Since the condition (3)
holds, $C(Y-f(U))=f(C(f^{-1}(Y-f(U))))=f(C(f^{-1}(Y)-f^{-1}f(U)))\subseteq
f(C(X-U))=f(X-U)$. By Proposition \ref{properties of subsets} (6), $%
I(f(U))=Y-C(Y-f(U))\supset Y-f(X-U)$. So $y\in Y-f(X-U)$ implies $y\in
I(f(U))$. This proves that $f$ is a $g$-pseudo-open mapping.
\end{proof}

It is not difficult to see that the Corollary below holds from Theorem \ref%
{characterizations of g pseudo open} and Proposition \ref{characterizations
of g continuous}. above.

\begin{corollary}
Let $f:(X,\mu )\longrightarrow (Y,\nu )$ be a mapping. Then the following
conditions are equivalent.

(1) $f$ is a $g$-pseudo-open mapping.

(2) If $B\subseteq Y$, then $C(B)=f(C(f^{-1}(B)))$.
\end{corollary}

Finally, we can establish relationships among $g-$open mappings, $g-$closed
mappings, $g-$pseudo-open mappings, and $g-$quotient mappings on generalized
topological spaces.

\begin{theorem}
\label{relationship}Let $f:(X,\mu )\longrightarrow (Y,\nu )$ be a mapping.
Consider the following conditions.

(1) $f$ is a $g$-open mapping.

(2) $f$ is a $g$-closed mapping.

(3) $f$ is a $g$-pseudo-open mapping.

(4) $f$ is a $g$-quotient mapping.

Then (1) $\Longrightarrow $ (3) $\Longrightarrow $ (4) and (2) $%
\Longrightarrow $ (3).
\end{theorem}

\begin{proof}
(1) $\Longrightarrow $ (3): Assume that $f$ is a $g$-open mapping. Let $y\in
Y$ such that $f^{-1}(y)\subseteq U\in \mu $. Then $y\in f(U)$ and $U=I(U)$.
By Theorem \ref{characterizations of g open}, $y\in f(U)=f(I(U))\subseteq
I(f(U))$. This proves that $f$ is a $g$-pseudo-open mapping.

(3) $\Longrightarrow $ (4): It holds from Theorem \ref{characterizations of
g closed}. Here, we give a proof by using characterizations of $g$%
-pseudo-open mappings and $g$-quotient mappings. Assume that $f$ is a $g$%
-pseudo-open mapping. Let $F\subseteq Y$ such that $f^{-1}(F)$ is a $\mu $%
-closed subset of $X$. Then $C(f^{-1}(F))=f^{-1}(F)$. By Theorem \ref%
{characterizations of g closed}, $C(F)=f(C(f^{-1}(F)))=f(f^{-1}(F))=F$. So $F
$ is a $\mu $-closed subset of $X$. This proves that $f$ is a $g$-quotient
mapping.

(2) $\Longrightarrow $ (3): Assume that $f$ is a $g$-closed mapping. Let $%
y\in Y$ such that $f^{-1}(y)\subseteq U\in \mu $. By Theorem \ref%
{characterizations of g closed}, there is $V\in \mu $ such that $%
f^{-1}(y)\subseteq V\subseteq U$ and $f(V)\in \nu $. So $y\in
f(V)=I(f(V))\subseteq I(f(U))$. This proves that $f$ is a $g$-pseudo-open
mapping.
\end{proof}

\begin{remark}
\label{reverse relationship} Among $g$-open mappings, $g$-closed mappings, $%
g $-pseudo-open mappings and $g$-quotient mappings, the only implications
that hold are those that can be obtained from Theorem \ref{relationship}. In
fact, if both $(X,\mu )$ and $(Y,\nu )$ in Theorem \ref{relationship} are
topological spaces, then $f$ is $g$-open (resp. $g$-closed, $g$-pseudo-open, 
$g$-quotient) if and only if $f$ is open (resp. closed, pseudo-open,
quotient). Moreover, among open mappings, closed mappings, pseudo-open
mappings and quotient mappings, all other implications rather than those
that are similar to Theorem \ref{relationship} cannot hold. (see \cite{E,L},
for example).
\end{remark}


\end{document}